\documentclass[a4paper,reqno]{amsart}

\usepackage[utf8]{inputenc}
\usepackage[english]{babel}

\usepackage{amsmath, amssymb, amsthm}
\usepackage{thmtools} % extends amsthm. for example \autoref (hyperref package required) is very handy
\usepackage{mathtools} % mehr mathmode symbols zb \mathclap fuer passendes spacing bei langem index, also for \DeclarePairedDelimiter

\usepackage{hyperref} % for clickable links inside the document
\hypersetup{hidelinks}
\usepackage[nameinlink]{cleveref} % for \cref and \Cref (has to be after amsmath)
\usepackage{dsfont} % for indicator function symbol (\mathds{1})
\usepackage{pifont} % for shaded square (\ding{113})
\usepackage{nicefrac} % for \nicefrac{a}{b}
\usepackage{upgreek}

\usepackage{ifthen}

\allowdisplaybreaks

\newcounter{i} % this counter is used for macros
\newtoks\striche % this toks is used for macros

\newcommand{\smatvdots}{\vphantom{\int\limits^x}\smash{\vdots}} % hack for better spacing in smallmatrix environment

\newcommand{\R}{\mathbb{R}}
\newcommand{\N}{\mathbb{N}} % natural numbers starting with 1
\newcommand{\C}{\mathbb{C}}

\newcommand{\sfC}{\mathsf{C}}

\newcommand{\Lp}[1]{\mathsf{L}^{#1}} % Lp spaces
\newcommand{\Hspace}{\mathsf{H}}
\newcommand{\iu}{\mathrm{i}} % imaginary unit
\newcommand{\diffd}{\mathrm{d}} % differential d
\newcommand{\dx}[1][x]{\,\diffd#1}

\newcommand{\cl}[2][]{\overline{#2}\ifthenelse{ \equal{#1}{} }{}{^{#1}}} % topological closure
\newcommand{\quspace}[3][]{#2/\mspace{1mu}#3} % quotient space
\DeclareMathOperator{\ran}{ran}

\DeclareMathOperator{\dom}{dom}

\DeclareMathOperator{\Div}{div}
\renewcommand{\div}{\Div}
\newcommand{\grad}{\nabla}

\DeclarePairedDelimiter{\set}{\{}{\}}
\DeclarePairedDelimiter{\norm}{\lVert}{\rVert}

\DeclarePairedDelimiterX{\dset}[2]{\{}{\}}{#1\;\delimsize\vert\;\mathopen{} #2}
\DeclarePairedDelimiterX{\scprod}[2]{\langle}{\rangle}{#1,#2}
\DeclarePairedDelimiterX{\dualprod}[2]{\langle}{\rangle}{#1,#2}

\renewcommand{\Re}{\operatorname{Re}}

\newcommand{\dualsymb}{\prime}

\newcommand{\adjunsymb}{\ast} % if you want to change the symbol of the adjoint just change this command

\newcommand{\adjun}[1][1]{%
  \setcounter{i}{1}%
  \striche={\adjunsymb}%
  \loop%
  \ifnum\value{i}<#1%
  \striche=\expandafter{\the\expandafter\striche\adjunsymb}%
  \stepcounter{i}%
  \repeat%
  ^{\the\striche}%
}

\newcommand{\dual}[1][1]{
  \setcounter{i}{1}%
  \striche={\dualsymb}%
  \loop%
  \ifnum\value{i}<#1%
  \striche=\expandafter{\the\expandafter\striche\dualsymb}%
  \stepcounter{i}%
  \repeat%
  ^{\the\striche}%
}

\newcommand{\mapping}[4]{%
  \left\{%
    \begin{array}{rcl}%
      #1 &\to & #2, \\
      #3 &\mapsto & #4
    \end{array}%
  \right.%
}

\newcommand{\hamiltonian}{\mathcal{H}}
\newcommand{\idop}{\mathrm{I}}
\newcommand{\opid}{\idop}

\newcommand{\XH}{\mathcal{X}_{\hamiltonian}}

\newcommand{\boundtr}[1][]{\gamma_{0}\ifthenelse{\equal{#1}{}}{}{^{#1}}}
\newcommand{\normaltr}[1][]{\gamma_{\nu}\ifthenelse{\equal{#1}{}}{}{^{#1}}}
\newcommand{\uu}{u}
\newcommand{\vv}{v}

\theoremstyle{plain}% default
\newtheorem{theorem}{Theorem}[section]
\newtheorem{lemma}[theorem]{Lemma}
\newtheorem{proposition}[theorem]{Proposition}
\newtheorem{corollary}[theorem]{Corollary}

\theoremstyle{definition}
\newtheorem{definition}[theorem]{Definition}

\theoremstyle{remark}
\newtheorem{remark}[theorem]{Remark}

\title[Stability of the wave equation in PHS modelling]{Stability of the multidimensional wave equation in port-Hamiltonian modelling}

\thanks{The authors are working in the ITN network ConFlex. This project is funded by the European Union’s Horizon 2020 research and innovation programme under the Marie Sklodowska-Curie grant agreement No 765579.}

\author[B. Jacob]{Birgit Jacob}
\thanks{Birgit Jacob is with the Fakult\"at f\"ur Mathematik und Naturwissenschaften, IMACM, Bergische Universit\"at Wuppertal, Germany \texttt{\href{mailto:bjacob@uni-wuppertal.de}{bjacob@uni-wuppertal.de}}}
\address{Fakult\"at f\"ur Mathematik und Naturwissenschaften,
Bergische Universit\"at Wuppertal, Germany}
\email[Birgit Jacob]{bjacob@uni-wuppertal.de}

\author[N. Skrepek]{Nathanael Skrepek}
\thanks{Nathanael Skrepek is with the Fakult\"at f\"ur Mathematik und Naturwissenschaften, IMACM,
Bergische Universit\"at Wuppertal, Germany \texttt{\href{mailto:skrepek@uni-wuppertal.de}{skrepek@uni-wuppertal.de}}}
\address{Fakult\"at f\"ur Mathematik und Naturwissenschaften,
Bergische Universit\"at Wuppertal, Germany}
\email[Nathanael Skrepek]{skrepek@uni-wuppertal.de}

\date{\today}

\begin{document}

\maketitle
\thispagestyle{empty}
\pagestyle{empty}

%%%%%%%%%%%%%%%%%%%%%%%%%%%%%%%%%%%%%%%%%%%%%%%%%%%%%%%%%%%%%%%%%%%%%%%%%%%%%%%%
\begin{abstract}
  We investigate the stability of the wave equation with spatial dependent coefficients on a bounded multidimensional domain.
  The system is stabilized via a scattering passive feedback law.
  We formulate the wave equation in a port-Hamiltonian fashion and show that the system is semi-uniform stable, which is a stability concept between exponential stability and strong stability.  Hence, this also implies strong stability of the system. In particular, classical solutions are uniformly stable.
  This will be achieved by showing that the spectrum of the port-Hamiltonian operator is contained in the left half plane $\C_{-}$ and the port-Hamiltonian operator generates a contraction semigroup. Moreover, we show that the spectrum consists of eigenvalues only and the port-Hamiltonian operator has a compact resolvent.
\end{abstract}

\section{Introduction}

% In this paper we investigate the stability of the wave equation
In this paper we investigate a stabilizing feedback for the following boundary control system
\begin{subequations}\label{eq:system-classical-formulation}
  \begin{align}
    \begin{aligned}
      u(t,\zeta) &= \frac{\partial w}{\partial T\nu} (t,\zeta),
      &&t\geq 0, \zeta\in \Gamma_{1},\\
      \frac{\partial^{2}w}{\partial t^{2}} (t,\zeta)
      &= \frac{1}{\rho(\zeta)} \div \left(T(\zeta) \grad w(t,\zeta)\right),
      &&t\geq 0, \zeta \in \Omega, \\
      w(t,\zeta)&= h(\zeta),
      &&t\geq 0, \zeta \in \Gamma_{0},  \\[1.2ex]
      w(0,\zeta) &= w_{0}(\zeta),
      &&\phantom{t\geq 0,\mathclose{}} \zeta \in \Omega,  \\
      \frac{\partial w}{\partial t} (0,\zeta) &= w_{1}(\zeta),
      &&\phantom{t\geq 0,\mathclose{}}\zeta \in \Omega, \\
      y(t,\zeta) &= \frac{\partial w}{\partial t} (t,\zeta),
      &&t\geq 0, \zeta\in \Gamma_{1},
    \end{aligned}
  \end{align}
  with feedback law
  \begin{align}\label{eq:system-feedback}
    \begin{aligned}
      %u(t) = -k y(t),
      \phantom{\frac{\partial^{2}w}{\partial t^{2}} (t,\zeta)}
      \mathllap{u(t,\zeta)}
      &=
      \mathrlap{-k(\zeta) y(t,\zeta),}
      \phantom{\frac{1}{\rho(\zeta)} \div \left(T(\zeta) \grad w(t,\zeta)\right),}
      &&t\geq 0, \zeta\in \Gamma_{1}, \\ % <-- linebreak is necessary for correct spacing
    \end{aligned}
  \end{align}
\end{subequations} 
where $u$ and $y$ are the boundary control and observation, respectively and $\Omega \subseteq \R^{n}$ is bounded domain with  Lipschitz boundary $\partial\Omega=\overline{\Gamma_0}\cup \overline{\Gamma_1}$ with $\Gamma_0\cap \Gamma_1=\emptyset$, $\Gamma_0$ and $\Gamma_1$ are open in the relative topology of $\partial\Omega$ and the boundaries of $\Gamma_0$ and $\Gamma_1$ have surface measure zero.  Note, that $\Gamma_{0}$ and $\Gamma_{1}$ do not have to be connected. %
%, see \Cref{fig1},
Furthermore, $w(\zeta,t)$ is the deflection at point $\zeta \in \Omega$ and $t\ge 0$, and  profile $h$ is given on $\Gamma_{0}$, where the wave is fixed.
Let Young's elasticity modulus
$T\colon \Omega \to \C^{n\times n}$ be a Lipschitz continuous matrix-valued function such that $T(\zeta)$ is a positive and invertible matrix (a.e.) and $T^{-1} \in \Lp{\infty}(\Omega)^{n\times n}$.
The vector $\nu$ denotes the outward normal at the boundary and $\frac{\partial}{\partial T\nu} w(t,\zeta) = T \nu \cdot \grad w(t,\zeta) = \nu \cdot T\grad w(t,\zeta)$ is the conormal derivative. Further, $k\colon \Gamma_{1} \to \R$ is a positive and bounded function such that also its pointwise inverse $k^{-1} = \frac{1}{k}$ is bounded.
Moreover, we have the Lipschitz continuous mass density $\rho \colon \Omega \to \R_{+}$, that satisfies $\rho^{-1} \in \Lp{\infty}(\Omega)$.
Finally, $w_{0}$ and $w_{1}$ are the initial conditions.

% \begin{figure}[h]
%   %\vspace{-1em}
%   \centering
%   \begin{tikzpicture}[scale=0.9]
%     \draw plot [smooth cycle] coordinates {(0,0) (1,1) (3,1) (2,0) (2,-1)};
%     \draw[-] (1.05,0.9) -- (0.95,1.1);
%     \draw[-] (1.4+0.05,-0.9+0.2) -- (1.4-0.05,-1.1+0.2);
%     % \node at (1.15,1.2) {$\nu$};
%     \node at (1.25,0.25) {$\Omega$};
%     \node[above] at (0,0.5) {$\Gamma_{0}$};
%     \node[right] at (2.2,0) {$\Gamma_{1}$};
%   \end{tikzpicture}
%   %\captionsetup[wrapfigure]{font=scriptsize}
%   \caption{Illustration of the Setting}\label{fig1}
% \end{figure}

Stability of~\eqref{eq:system-classical-formulation} has been studied in the literature by several authors, see e.g.~\cite{bar92,humaloja-kurula-paunonen,lag83,quinn-russell-1977}. Strong stability  has been investigated in~\cite{quinn-russell-1977}. Further, exponential stability of the wave equation with constant $T$ and $\rho$ has been shown in \cite{lag83} using multiplier methods. For smooth domains, in \cite{bar92} the equivalence of exponential stability and the so-called \emph{geometric control condition} was shown by methods from micro-local analysis. % here piecewise $C^{\infty}$ boundary is needed.
In~\cite{humaloja-kurula-paunonen} this system also appears in port-Hamiltonian formulation, but with constant $T$ and $\rho$ and $C^{2}$ boundary. Under these restriction it could be shown that this systems is even exponential stable. However, semi-uniform stability, a notion which is stronger than strong stability and weaker than exponential stability, of the multidimensional wave equation with spatial dependent functions $\rho$ and $T$ on quite general domains has not been studied in the literature.

We aim to show semi-uniform stability of the multidimensional wave equation~\eqref{eq:system-classical-formulation} using a port-Hamiltonian formulation. 
Semi-uniform stability implies strong stability, and thus we extend the results obtained in~\cite{quinn-russell-1977}.  To prove our main result we use the fact that semi-uniform stability is satisfied if the port-Hamiltonian operator generates a contraction semigroup and possesses no spectrum in the closed right half plane.
%In order to show semi-uniform stability  we use the port-Hamiltonian formulation of the wave equation. 
Port-Hamiltonian systems encode the underlying physical principles such as conservation laws directly into the structure of the system structure. For 
finite-dimensional systems there is by now a well-established theory
\cite{vanDerSchaft06,EbMS07,DuinMacc09}. The port-Hamiltonian approach has been further extended to the  
infinite-dimensional situation, see e.g.~\cite{VanDerSchaftMaschke_2002,MM05,JeSc09,ZwaGorMasVil10,Villegas_2007,JacZwa12,kurula-zwart-wave}. 
In~\cite{kurula-zwart-wave} the authors showed that the port-Hamiltonian formulation of the wave equation~\eqref{eq:system-classical-formulation} in $n$ spatial dimensions possess unique mild and classical solutions.

We proceed as follows. In \Cref{sec2} we model the multidimensional wave equation as a port-Hamiltonian system with a suitable state space. The main results concerning stability are then obtained in \Cref{sec3}, where we analyze the spectrum of the differential operator of the port-Hamiltonian formulation. We will see that finding points in the resolvent set is linked to solvability of lossy Helmholtz equations. We will show that our operator has a compact resolvent and its resolvent set contains the imaginary axis. At that point we can apply existing theory to justify semi-uniform stability.
Finally, used notations and results on Sobolev spaces and G{\aa}rdings inequalities are presented in the Appendix.

\section{Port-Hamiltonian formulation of the System}\label{sec2}
In order to find a port-Hamiltonian formulation of our system, that is suitable for our purpose, we split the system~\eqref{eq:system-classical-formulation} into a time independent system for the equilibrium and a dynamical system with homogeneous boundary conditions. The time static system for the equilibrium is given by
\begin{align}\label{eq:equilibrium-system}
  \begin{aligned}
    \div T(\zeta) \grad w_{\mathrm{e}}(\zeta)
    &= 0,
    && \zeta \in \Omega,\\[1.2ex]
    w_{\mathrm{e}}(\zeta) &= h(\zeta), &&\zeta \in \Gamma_{0}, \\
    \frac{\partial w_{\mathrm{e}}}{\partial T \nu} (\zeta) &= 0, &&\zeta \in \Gamma_{1},
  \end{aligned}
\end{align}
and a dynamical system with homogeneous Dirichlet boundary conditions on $\Gamma_{0}$ is given by
\begin{align}\label{eq:dynamical-system}
  \begin{aligned}
    \frac{\partial^{2}w_{\mathrm{d}}}{\partial t^{2}} (t,\zeta)
    &= \frac{1}{\rho(\zeta)} \div (T(\zeta) \grad w_{\mathrm{d}}(t,\zeta)),
    &&t\geq 0, \zeta \in \Omega, \\
    w_{\mathrm{d}}(t,\zeta)&=0,
    &&t\geq 0, \zeta \in \Gamma_{0}, \\
    w_{\mathrm{d}}(0,\zeta) &= w_{0}(\zeta) - w_{\mathrm{e}}(\zeta),
    &&\phantom{t\geq 0,\mathclose{}} \zeta \in \Omega, \\
    \frac{\partial w_{\mathrm{d}}}{\partial t}(0,\zeta) &= w_{1}(\zeta),
    &&\phantom{t\geq 0,\mathclose{}} \zeta \in \Omega, \\
    \frac{\partial w_{\mathrm{d}}}{\partial t} (t,\zeta)
    &= -k\frac{\partial w_{\mathrm{d}}}{\partial T\nu} (t,\zeta),
    &&t\geq 0, \zeta\in \Gamma_{1}.
  \end{aligned}
\end{align}
The original system is solved by $w(t,\zeta) = w_{\mathrm{e}}(t,\zeta) + w_{\mathrm{d}}(\zeta)$.
As in~\cite{kurula-zwart-wave} the system in~\eqref{eq:dynamical-system} can be described in a port-Hamiltonian manner by choosing
the state
$
x(t,\zeta)
=
\left[
  \begin{smallmatrix}
    \rho(\zeta) \frac{\partial}{\partial t} w_{\mathrm{d}}(t,\zeta) \\
    \grad w_{\mathrm{d}}(t,\zeta)
  \end{smallmatrix}
\right]
$. By using the convention 
\begin{equation*}
  \begin{bmatrix}
    x_1(t) \\ x_2(t)
  \end{bmatrix}
  \coloneqq x(t) \coloneqq x(t,\cdot)
\end{equation*}
 we can write the system~\eqref{eq:dynamical-system} as
\begin{align*}
  \frac{\diffd}{\diffd t} x(t)
  &=
  \begin{bmatrix}
    0 & \div \\
    \grad & 0
  \end{bmatrix}
            \begin{bmatrix}
              \frac{1}{\rho} & 0 \\
              0 & T
            \end{bmatrix}
                  x(t),
  \\
  x(0) &=
         \begin{bmatrix}
           \rho w_{1} \\
           \grad (w_{0} - w_{\mathrm{e}})
         \end{bmatrix},
  \\
  \boundtr \tfrac{1}{\rho} x_{1}(t)\big\vert_{\Gamma_{0}} &= 0, \\
  \boundtr \tfrac{1}{\rho} x_{1}(t)\big\vert_{\Gamma_{1}} &= -k \normaltr T x_{2}(t)\big\vert_{\Gamma_{1}}.
\end{align*}
The boundary trace $\boundtr$ and the normal trace $\normaltr$ are explained in the appendix.
Kurula and Zwart~\cite{kurula-zwart-wave} choose the state space $\Lp{2}(\Omega)^{n+1}$ equipped with the energy inner product
\begin{align*}
  \scprod{x}{y}
  \coloneqq \scprod*{x}{
  \begin{bsmallmatrix}
    \frac{1}{\rho} & 0 \\
    0 & T
  \end{bsmallmatrix}
        y
        }_{\Lp{2}(\Omega)^{n+1}},
\end{align*}
which is equivalent to the standard inner product of $\Lp{2}(\Omega)^{n+1}$ thanks to the assumptions on $T$ and $\rho$.
They then show the existence of mild and classical solution via semigroup methods. For well-posedness this is a suitable state space, but when it comes to stability this state space is too large as it does not reflect the fact that the second component of the state variable $x_2$ is of the form $\grad v$, for some function $v$ in the Sobolev space $\Hspace^{1}_{\Gamma_{0}}(\Omega)$. For the precise definition of $\Hspace^{1}_{\Gamma_{0}}(\Omega)$ we refer the reader to the appendix.
Thus, we choose the state space $\XH$ as $\Lp{2}(\Omega) \times \grad \Hspace^{1}_{\Gamma_{0}}(\Omega)$, instead of $\Lp{2}(\Omega)^{n+1}$.
Note that $\grad \Hspace^{1}_{\Gamma_{0}}(\Omega)$ is closed in $\Lp{2}(\Omega)^{n}$ by Poincar{\'e}'s inequality.
Hence, $\XH$ is also a Hilbert space with the $\Lp{2}$-inner product. Nevertheless, we also use the  equivalent energy inner product on $\XH$, that is
\begin{align*}
  \scprod{x}{y}_{\XH}
  \coloneqq \scprod*{x}{
  \begin{bsmallmatrix}
    \frac{1}{\rho} & 0 \\
    0 & T
  \end{bsmallmatrix}
        y
        }_{\Lp{2}(\Omega)^{n+1}}.
\end{align*}
Furthermore, we define
\begin{align*}
  \mathfrak{A}
  &\coloneqq
    \left[
    \begin{smallmatrix}
      0 & \div \\
      \grad & 0
    \end{smallmatrix}
              \right]
              \left[
              \begin{smallmatrix}
                \frac{1}{\rho} & 0 \\
                0 & T
              \end{smallmatrix}
                    \right]
  \\
  \mathllap{\text{with}\quad}\dom(\mathfrak{A})
  &\coloneqq
  \left[
  \begin{smallmatrix}
    \frac{1}{\rho} & 0 \\
    0 & T
  \end{smallmatrix}
        \right]^{-1}
        \big(\Hspace^{1}_{\Gamma_{0}}(\Omega) \times \Hspace(\div,\Omega)\big)
\end{align*}
as densely defined operator on $\Lp{2}(\Omega)^{n+1}$. The definition of $ \Hspace(\div,\Omega)\big)$ is given in the appendix. Note that we have already packed the boundary condition $\boundtr \tfrac{1}{\rho} x_{1} = 0$ on $\Gamma_{0}$ into the domain of $\mathfrak{A}$.
Moreover, by construction $\ran \mathfrak{A} = \XH$.
Taking the state space and the remaining boundary conditions (feedback) into account gives
\begin{align}
  \label{eq:def-A}
  \begin{split}
    A &\coloneqq \mathfrak{A}\big\vert_{\dom(A)}, \\
    \mathllap{\text{where}\quad}
    \dom(A)
    &\coloneqq
    \dset*{x \in \dom (\mathfrak{A})}%
    {\boundtr \tfrac{1}{\rho} x_{1} = -k \normaltr T x_{2} \text{ on } \Gamma_{1}}
    \cap \XH
  \end{split}
\end{align}
as an operator on $\XH$. Note that $\ran A \subseteq \ran \mathfrak{A} = \XH$. Therefore the operator $A$ indeed maps into $\XH$.

The corresponding operator on $\Lp{2}(\Omega)^{n+1}$ would be
\begin{align}
  \label{eq:def-A_0}
  \begin{split}
    A_{0} &\coloneqq \mathfrak{A}\big\vert_{\dom (A_{0})}, \\
    \mathllap{\text{where}\quad}
    \dom(A_{0})
    &\coloneqq
    \dset*{x \in \dom(\mathfrak{A})}%
    {\boundtr \tfrac{1}{\rho} x_{1} = -k \normaltr T x_{2} \text{ on } \Gamma_{1}}.
  \end{split}
\end{align}
By~\cite{kurula-zwart-wave}, $A_{0}$ generates a contraction semigroup on $\Lp{2}(\Omega)^{n+1}$ endowed with
$\scprod{x}{y} \coloneqq \scprod*{x}{\begin{bsmallmatrix} \frac{1}{\rho} & 0 \\ 0 & T\end{bsmallmatrix} y}_{\Lp{2}}$.
Note that this operator allows elements in its domain which do not respect that the second component is a gradient field.
This can lead to solutions that are not related to the original problem anymore, as by construction of the state $x(t,\zeta)$ the second component is $\grad w_{\mathrm{d}}(t,\zeta)$ and therefore a gradient field.
\Cref{th:zero-eigenvalue-of-A0} shows that this is problematic for stability.

We do not need to rebuild the semigroup theory in~\cite{kurula-zwart-wave} for the state space $\XH$. We will see that $A$ inherits most of the properties of $A_{0}$ as $A = A_{0}\big\vert_{\XH}$.

\begin{lemma}\label{le:invariant-subspace}
  Let $(T(t))_{t\ge 0}$ be a strongly continuous semigroup on a Hilbert space $X$ and $\tilde{A}$ its generator. Then
  every subspace $V \supseteq \ran \tilde{A}$ is invariant under $(T(t))_{t\ge 0}$.

  Moreover, $\tilde{A}\big\vert_{V}$ generates the strongly continuous semigroup
  \begin{equation*}
    (T_{V}(t))_{t\ge 0} \coloneqq (T(t)\big\vert_{V})_{t\ge 0},
  \end{equation*}
  if $V$ is additionally closed.
\end{lemma}

\begin{proof}
  Let $t \geq 0$ and $x\in V$. Then it is well-known that
  \begin{align*}
    \underbrace{\tilde{A} \int_{0}^{t} T(s) x \dx[s]}_{\mathrlap{\in \ran \tilde{A}\subseteq V}} = T(t) x - \underbrace{x}_{\mathrlap{\in V}}.
  \end{align*}
  Hence, $T(t)x \in V$, because the left-hand-side is in
  $\ran \tilde{A} \subseteq V$ and $V$ is a subspace.
The remaining assertion follows from~\cite[ch.\ II sec.\ 2.3]{engel-nagel}.
% Invariant subspace of the semigroup => generator is defined by its restriction on this subspace
\end{proof}

\begin{remark}
  If the strongly continuous semigroup $(T(t))_{t\ge 0}$ is even a contraction semigroup, then also $(T_{V}(t))_{t\ge 0}$ is a contraction semigroup.
\end{remark}

\begin{proposition}
  The operator $A$ given by~\eqref{eq:def-A} is a generator of contraction semigroup.
\end{proposition}

\begin{proof}
  By~\cite{kurula-zwart-wave}, $A_{0}$ (defined in~\eqref{eq:def-A_0})
  is a generator of a contraction semigroup $(T_{0}(t))_{t\ge 0}$.
  Because of $\ran A_{0} \subseteq \ran \mathfrak{A} = \XH$ and
  \Cref{le:invariant-subspace}
  $A = A_{0}\big\vert_{\XH}$ generates the contraction semigroup $(T(t))_{t\ge 0} \coloneqq (T_{0}(t)\big\vert_{\XH})_{t\ge 0}$.
\end{proof}

The following lemma is an easy consequence of the integration by parts formula for $\div$-$\grad$ from the appendix and will be useful in the next section.

\begin{lemma}\label{le:boundary-triple-property}
  Let $A$ be given by~\eqref{eq:def-A} and $x,y \in \XH$. Then
  \begin{equation*}
    \scprod{Ax}{y}_{\Lp{2}(\Omega)} + \scprod{x}{Ay}_{\Lp{2}(\Omega)}
    = \scprod{\normaltr T x_{2}}{\boundtr \tfrac{1}{\rho} y_{1}}_{\Lp{2}(\Gamma_{1})}
    + \scprod{\boundtr \tfrac{1}{\rho} x_{1}}{\normaltr T y_{2}}_{\Lp{2}(\Gamma_{1})}.
  \end{equation*}
  And in particular
  \begin{align*}
    \Re \scprod{Ax}{x}_{\Lp{2}(\Omega)}
    = \Re \scprod{\normaltr T x_{2}}{\boundtr \tfrac{1}{\rho} x_{1}}_{\Lp{2}(\Gamma_{1})}.
  \end{align*}
\end{lemma}

\section{Stability Results}\label{sec3}

In this section we prove semi-uniform stability of the multidimensional wave equation~\eqref{eq:system-classical-formulation}. We start with the definition of semi-uniform stability and strong stability.

\begin{definition}
  We say a strongly continuous semigroup $(T(t))_{t\ge 0}$ on a Hilbert space $X$ is \emph{strongly stable}, if for every $x \in X$
  \begin{align*}
    \lim_{t\to\infty} \norm{T(t)x}_{X} = 0.
  \end{align*}

  We say a continuous semigroup $(T(t))_{t\geq 0}$ on a Hilbert space $X$ is \emph{semi-uniform stable},
  if there exists a continuous  monotone decreasing function
  $f\colon [0,\infty) \to [0,\infty)$ with $\lim_{t\to \infty} f(t)=0$ and
  \[
    \norm{T(t)x}_{X} \leq f(t) \norm{x}_{\dom (A)}
  \]
  for every $x \in \dom(A)$.
\end{definition}

\begin{remark}
  Note that in~\cite[sec.~3]{chill-seifert-tomilov-2020} semi-uniform stability is defined by $\norm{T(t)A^{-1}} \to 0$, where $A$ is the generator of $(T(t))_{t\geq 0}$. It can be easily seen that this is equivalent to our definition.

  Moreover, in~\cite[sec.\ 3]{chill-seifert-tomilov-2020} it is explained that semi-uniform stability is a concept between exponential stability and strong stability. In particular, semi-uniform stability implies strong stability.
\end{remark}

%Hence, we will only focus on proving semi-uniform stability,

The already mentioned article~\cite{chill-seifert-tomilov-2020} is an overview article on semi-uniform stability.
We remark that this notion is sometimes called differently, e.g.\ in~\cite{su-tucsnak-weiss-2020} it is called uniform stability for smooth data (USSD).

In the following we denote by $A$ the operator given by~\eqref{eq:def-A} which is associated to the port-Hamiltonian formulation of~\eqref{eq:system-classical-formulation}.

Our main result is the following theorem.

\begin{theorem}\label{th:semigroup-semi-uniform-stable}
  The semigroup generated by $A$ is semi-uniform stable.
\end{theorem}

The \hyperref[prf:semigroup-semi-uniform-stable]{proof of \Cref*{th:semigroup-semi-uniform-stable}}
is given at the end of the section.

\begin{remark}
  For the original system~\eqref{eq:system-classical-formulation} strong stability of $A$ translates to: There is a $w_{\mathrm{e}} \in \Hspace^{1}(\Omega)$ such that for every initial value
  $w_{0} \in \Hspace^{1}(\Omega)$,
  $w_{1}\in \Lp{2}(\Omega)$
  the corresponding solution $w$ satisfies
  \begin{align*}
    \lim_{t\to\infty} \norm{w(t,\cdot) - w_{\mathrm{e}}(\cdot)}_{\Hspace^{1}(\Omega)} = 0.
  \end{align*}
\end{remark}

We will make use of a characterization of semi-uniform stability in~\cite[Theorem 3.4]{chill-seifert-tomilov-2020} to show that $A$, given by~\eqref{eq:def-A}, generates a semi-uniform stable semigroup. As $A$ generates a bounded strongly continuous semigroup, by this theorem a sufficient condition for semi-uniform stability is given by $\upsigma(A) \cap \iu \R = \emptyset$.
Here $\upsigma(A)$ denotes the spectrum of the operator $A$.
Hence, it suggests itself to analyse the spectrum of $A$ or its complement in $\mathbb C$, the resolvent set.

We will show that calculating the resolvent set $\uprho(A)$ is related to a lossy Helmholtz problem: Find a function $u\colon \Omega \to \C$ that satisfies
\begin{align}\label{eq:lossy-helmholtz}
  \begin{split}
    \div T \grad \uu - \lambda^{2} \rho \uu &= f
    \quad\text{in}\quad \Omega,
    \\
    \tfrac{\partial}{\partial T\nu} \uu + \lambda k^{-1}  \uu &= \mathrlap{g}\phantom{f}
    \quad\text{on}\quad \Gamma_{1},
  \end{split}
\end{align}
where $\lambda \in \C\setminus\set{0}$, $f \in \Lp{2}(\Omega)$, $g \in \Lp{2}(\Gamma_{1})$, and $k$, $\rho$ and $T$ are the functions from the beginning.
A weak formulation of this problem can be derived by taking the inner product with $\vv \in \Hspace^{1}_{\Gamma_{0}}(\Omega)$, apply an integration by parts formula for $\div$-$\grad$ and taking the boundary conditions into account:
\begin{align}
  \label{eq:weak-helmholtz-formulation}
  \begin{split}
    \MoveEqLeft[16]
    \scprod{T \grad \uu}{\grad \vv}_{\Lp{2}(\Omega)}
    + \lambda^{2} \scprod{\rho \uu}{\vv}_{\Lp{2}(\Omega)}
    + \lambda \scprod{k^{-1}\boundtr \uu}{\boundtr \vv}_{\Lp{2}(\Gamma_{1})}\\
    &= \scprod{-f}{\vv}_{\Lp{2}(\Omega)} + \scprod{g}{\boundtr \vv}_{\Lp{2}(\Gamma_{1})}.
  \end{split}
\end{align}
We define
\begin{align*}
  b(\uu,\vv)
  &\coloneqq
    \scprod{T \grad \uu}{\grad \vv}_{\Lp{2}(\Omega)}
    + \lambda^{2} \scprod{\rho \uu}{\vv}_{\Lp{2}(\Omega)}
    + \lambda \scprod{ k^{-1}\boundtr \uu}{\boundtr \vv}_{\Lp{2}(\Gamma_{1})}
  \\
  F(\vv)
  &\coloneqq
    \scprod{-f}{\vv}_{\Lp{2}(\Omega)} + \scprod{g}{\boundtr \vv}_{\Lp{2}(\Gamma_{1})},
\end{align*}
so that we can write~\eqref{eq:weak-helmholtz-formulation} as
\begin{align}\label{eq:weak-helmholtz-short}
  b(\uu,\vv) = F(\vv).
\end{align}
A \textit{weak solution} of~\eqref{eq:lossy-helmholtz} is a function $\uu \in \Hspace^{1}_{\Gamma_{0}}(\Omega)$ that satisfies~\eqref{eq:weak-helmholtz-short} for every $\vv \in \Hspace^{1}_{\Gamma_{0}}(\Omega)$.

\begin{lemma}\label{le:properties-of-weak-solution}
  Let $\uu$ be a weak solution of the Helmholtz problem~\eqref{eq:lossy-helmholtz}.
  Then $\uu \in \Hspace^{1}_{\Gamma_{0}}(\Omega)$, $T\grad\uu \in \Hspace(\div,\Omega)$ and
  in particular,
  \begin{align*}
    \div T \grad \uu - \lambda^{2} \rho \uu &= f
    \quad\text{in}\quad \Lp{2}(\Omega), \\
    \normaltr T \grad \uu + \lambda k^{-1} \boundtr \uu &= \phantom{f}\mathllap{g}
    \quad\text{in}\quad \Lp{2}(\Gamma_{1}).
  \end{align*}
\end{lemma}

\begin{proof}
  A weak solution $\uu$ is by definition in $\Hspace^{1}_{\Gamma_{0}}(\Omega)$ and satisfies $b(\uu,\vv) = F(\vv)$ for all $v \in \Hspace^{1}_{\Gamma_{0}}(\Omega)$. If we choose $\vv \in \sfC_{\mathrm{c}}^{\infty}(\Omega)$, then all boundary integrals vanish. Hence,
  \begin{align*}
    \scprod{T\grad \uu}{\grad \vv}_{\Lp{2}(\Omega)}
    = \scprod{-f}{\vv}_{\Lp{2}(\Omega)} - \lambda^{2}\scprod{\rho \uu}{\vv}_{\Lp{2}(\Omega)},
  \end{align*}
  which implies that $T\grad \uu \in \Hspace(\div,\Omega)$ and
  $\div T\grad \uu = f + \lambda^{2}\rho \uu$.
  Using this and choosing again $\vv \in \Hspace^{1}_{\Gamma_{0}}(\Omega)$ in the weak formulation gives
  \begin{equation*}
    \dualprod{\normaltr T \uu}{\boundtr \vv}_{\Hspace^{-\nicefrac{1}{2}}(\Gamma_{1}),\Hspace^{\nicefrac{1}{2}}(\Gamma_{1})}
    + \lambda \scprod{k^{-1} \boundtr \uu}{\boundtr \vv}_{\Lp{2}(\Gamma_{1})}
    = \scprod{g}{\boundtr \vv}_{\Lp{2}(\Gamma_{1})}.
  \end{equation*}
  Therefore, $\normaltr T \uu$ has an $\Lp{2}(\Gamma_{1})$ representative and
  $\normaltr T \uu + \lambda k^{-1} \boundtr \uu = g$.
\end{proof}

Note that for
$y =
\left[
  \begin{smallmatrix}
    y_{1} \\ y_{2}
  \end{smallmatrix}
\right]
\in \XH
$
there exists a $\phi \in \Hspace^{1}_{\Gamma_{0}}(\Omega)$ such that $y_{2} = \grad \phi$.
This $\phi$ continuously depends on $y_{2}$ by Poincar\'{e}'s inequality. If $\Gamma_{0} = \emptyset$, then we choose $\phi \in \quspace{\Hspace^{1}(\Omega)}{\R}$ ($\phi \in \Hspace^{1}(\Omega)$ and $\int_{\Omega} \phi \dx[\uplambda] = 0$) for uniqueness and continuity.

\begin{lemma}\label{le:resolvent-set-equivlant-to-helmholtz}
  Let $A$ be the operator defined in~\eqref{eq:def-A}.
  Then $\lambda \in \uprho(A)\setminus\set{0}$ is equivalent to: The system
  \begin{align}\label{eq:weak-helmholtz-system-for-resolvent-set}
    \begin{split}
      \div T\grad \uu - \lambda^{2} \rho \uu
      &= \lambda y_{1} + \lambda^{2} \rho \phi
      \quad\text{in}\quad \Omega,
      \\
      \tfrac{\partial}{\partial T\nu} \uu +  \lambda k^{-1} \uu
      &=
      \mathrlap{- \lambda k^{-1} \phi}
      \phantom{\lambda y_{1} + \lambda^{2} \rho \phi}
      \quad\text{on}\quad \Gamma_{1},
    \end{split}
  \end{align}
  is weakly solvable for every
  $y =
  \left[
    \begin{smallmatrix}
      y_{1} \\ y_{2}
    \end{smallmatrix}
  \right]
  \in \XH
  $, where $\phi$ is defined by $\grad \phi = y_{2}$ as described above.
\end{lemma}

\begin{proof}
  For $\lambda \in \uprho(A)\setminus\set{0}$ and $y \in \XH$ there exists an $x \in \dom (A)$ such that
  $(A - \lambda) x = y$.
  Hence,
  \begin{align*}
    \div T x_{2} - \lambda x_{1} &= y_{1} \\
    \grad \tfrac{1}{\rho} x_{1} - \lambda x_{2} &= \grad \phi
    \quad \Rightarrow \quad
    x_{2} = \tfrac{1}{\lambda} \grad (\tfrac{1}{\rho} x_{1} - \phi).
  \end{align*}
  Substituting $x_{2}$ in the first equation, multiplying by $\lambda$ and adding $\lambda^{2}\rho\phi$ on both sides yields
  \begin{align*}
    \div T \grad (\tfrac{1}{\rho} x_{1} - \phi) - \lambda^{2} \rho (\tfrac{1}{\rho} x_{1} - \phi)
    = \lambda y_{1} + \lambda^{2}\rho\phi.
  \end{align*}
  Since $x\in\dom (A)$ we have
  $k \normaltr T x_{2} + \boundtr \tfrac{1}{\rho} x_{1} = 0$ which becomes
  \begin{align*}
    \normaltr T \grad(\tfrac{1}{\rho} x_{1} - \phi)
    + \lambda k^{-1} \boundtr (\tfrac{1}{\rho}x_{1} - \phi)
    =-\lambda k^{-1} \boundtr \phi.
  \end{align*}
  Hence, $\uu \coloneqq (\tfrac{1}{\rho}x_{1} - \phi)$ is a weak solution of the system~\eqref{eq:weak-helmholtz-system-for-resolvent-set}.
On the other hand if $\uu$ is a weak solution of~\eqref{eq:weak-helmholtz-system-for-resolvent-set},
  then
  $x \coloneqq
  \left[
    \begin{smallmatrix}
      \rho (\uu + \phi) \\
      \frac{1}{\lambda} \grad \uu
    \end{smallmatrix}
  \right]
  \in \dom (A)
  $
  and $(A-\lambda)x = y$ by
  \Cref{le:properties-of-weak-solution}.
\end{proof}

\begin{theorem}\label{th:helmholtz-weakly-solvable}
  For every $\lambda \in \iu \R \setminus \set{0}$ the
  system~\eqref{eq:weak-helmholtz-system-for-resolvent-set}
  is weakly solvable.
\end{theorem}

\begin{proof}
  We set $\lambda = \iu \eta$, where $\eta \in \R \setminus \set{0}$.

  Note that by
  % melenk lecture notes
  % https://www.asc.tuwien.ac.at/~melenk/teach/fem_WS1617/part6.ps
  \begin{align*}
    \Re b(\uu, \uu)
    = \norm{T^{\nicefrac{1}{2}} \grad \uu}_{\Lp{2}(\Omega)}^{2}
    - \eta^{2} \norm{\rho^{\nicefrac{1}{2}} \uu}_{\Lp{2}(\Omega)}^{2},
  \end{align*}
  $b(\cdot,\cdot)$ satisfies a G{\aa}rding inequality
  (see \Cref{def:garding-inequality}).

  By G{\aa}rding's inequality
  it is sufficient to show that $b(\cdot,\cdot)$ is a non-de\-ge\-nerated sesquilinear form, (see e.g.\
  \Cref{th:fredholm-alternative-garding}).
  Suppose there is a $\uu \in \Hspace^{1}_{\Gamma_{0}}(\Omega)$ such that $b(\uu,\vv) = 0$ for all $\vv \in \Hspace^{1}_{\Gamma_{0}}(\Omega)$. Then $b(\uu,\uu) = 0$ and by
  separating the imaginary part we have
  \begin{align*}
    % \scprod{T \grad w}{\grad w} - \eta^{2} \scprod{\rho w}{w} &= 0 \\
    \iu \eta k\scprod{\boundtr \uu}{\boundtr \uu}_{\Lp{2}(\Gamma_{1})} &= 0.
  \end{align*}
  Hence, $\uu \in \Hspace^{1}_{0}(\Omega)$. Moreover, $\uu$ is a weak solution of the corresponding system to $b(\uu,\vv)=\tilde{F}(\vv)$, where $\tilde{F}(\vv) \coloneqq 0$.
  By
  \Cref{le:properties-of-weak-solution},
  $\div T \grad \uu + \eta^{2}\rho \uu = 0$ in $\Lp{2}(\Omega)$ and $\normaltr T \uu = 0$ in $\Lp{2}(\Gamma_{1})$.
  Summed up $\uu$ satisfies
  \begin{align*}
    \div T \grad \uu + \eta^{2} \rho \uu &= 0, \\
    \boundtr \uu &= 0, \\
    \normaltr T \uu\big\vert_{\Gamma_{1}} &= 0.
  \end{align*}

  By the unique continuation principle (see e.g.~\cite[Theorem~1.7, Remark~1.8]{UCP-TS}), $\uu$ has to be $0$ and consequently $b(\cdot,\cdot)$ is non-degenerated.
\end{proof}

\begin{remark}
  The system~\eqref{eq:weak-helmholtz-system-for-resolvent-set} is also solvable for $\lambda \in \C_{+}$, but we already knew from the dissipativity of $A$ that $\C_{+} \subseteq \uprho(A)$.
\end{remark}

\begin{corollary}\label{th:i-axis-without-0-resolvent}
  $\iu\R \setminus \set{0} \cup \C_{+}\subseteq \uprho(A)$.
\end{corollary}

\begin{proof}
  This is a direct consequence of
  \Cref{le:resolvent-set-equivlant-to-helmholtz}
  and
  \Cref{th:helmholtz-weakly-solvable}.
\end{proof}

\begin{lemma}\label{th:imaginary-ev-boundary-condition}
  If $\lambda \in \iu \R$ is an eigenvalue of $A$, then a corresponding
  eigenvector $x$ satisfies $\normaltr T x_{2}\big\vert_{\Gamma_{1}} = \boundtr \tfrac{1}{\rho} x_{1}\big\vert_{\Gamma_{1}} = 0$.
\end{lemma}

\begin{proof}
  By \Cref{le:boundary-triple-property} we have
  \begin{align*}
    \Re \scprod{(A-\lambda)x}{x}_{\Lp{2}(\Omega)}
    &= \Re \scprod{A x}{x}_{\Lp{2}(\Omega)} - \Re \lambda\scprod{x}{x}_{\Lp{2}(\Omega)} \\
    &= \Re \scprod{\normaltr T x_{2}}{\boundtr \tfrac{1}{\rho} x_{1}}_{\Lp{2}(\Gamma_{1})}
      - \Re \lambda \norm{x}_{\Lp{2}(\Omega)}^{2}\\
    &= -k \Re \scprod{\boundtr \tfrac{1}{\rho} x_{1}}{\boundtr \tfrac{1}{\rho} x_{1}}_{\Lp{2}(\Gamma_{1})}
      - \Re \lambda \norm{x}_{\Lp{2}(\Omega)}^{2}\\
    &= -k \norm{\boundtr \tfrac{1}{\rho} x_{1}}^{2}_{\Lp{2}(\Gamma_{1})} - \Re \lambda \norm{x}_{\Lp{2}(\Omega)}^{2}.
  \intertext{If $x$ is an eigenvector of $\lambda \in \iu \R$, then this equation becomes}
    0 &= -k \norm{\boundtr \tfrac{1}{\rho} x_{1}}^{2}_{\Lp{2}(\Gamma_{1})},
  \end{align*}
  which also gives $\normaltr T x_{2}\big\vert_{\Gamma_{1}} = 0$ by the boundary conditions.
\end{proof}

\begin{lemma}\label{th:zero-no-eigenvalue}
  Let $A: \dom(A) \subseteq \XH \to \XH$ be the operator from the beginning.
  Then $0$ is not an eigenvalue of $A$. %, i.e.\
%  $0 \notin \upsigma_{\mathrm{p}}(A)$.
\end{lemma}

\begin{proof}
  Let us assume $0$ that is an eigenvalue of $A$ and $x$ be an eigenvector.
  Then $\div T x_{2} = 0$ and
  $\grad \tfrac{1}{\rho}x_{1} = 0$ and
  by
  \Cref{th:imaginary-ev-boundary-condition}
  $x$ satisfies
  $\normaltr T x_{2}\big\vert_{\Gamma_{1}} = 0 = \boundtr \tfrac{1}{\rho} x_{1}\big\vert_{\Gamma_{1}}$.
  Hence, for arbitrary $f\in\Hspace^{1}_{\Gamma_{0}}(\Omega)$ we have
  \begin{align*}
    0 = \scprod{\div T x_{2}}{f}_{\Lp{2}} = - \scprod{T x_{2}}{\grad f}_{\Lp{2}},
  \end{align*}
  which implies $T x_{2} \perp \grad \Hspace^{1}_{\Gamma_{0}}(\Omega)$. Since by assumption
  $T x_{2} \in \grad \Hspace^{1}_{\Gamma_{0}}(\Omega)$ we conclude $x_{2} = 0$.
  Finally, $x_{1} = 0$ by Poincar\'e's inequality.
  Therefore, $0$ cannot not be an eigenvalue.
\end{proof}

% compact embedding
\begin{theorem}\label{th:compact-embedding-state-space}
  Let
  \begin{align*}
    X
    &\coloneqq
      \grad \Hspace^{1}_{\Gamma_{0}}(\Omega)
      \cap
      \dset{f \in \Hspace(\div,\Omega)}{\normaltr f\big\vert_{\Gamma_{1}} \in \Lp{2}(\Gamma_{1})}
    \\
    \mathllap{\text{with}\quad}
    \norm{f}_{X}
    &\coloneqq
      \sqrt{%
      \norm{f}_{\Lp{2}(\Omega)^{n}}^{2}
      + \norm{\div f}_{\Lp{2}(\Omega)}^{2}
      + \norm{\normaltr f}_{\Lp{2}(\Gamma_{1})}^{2}%
      }.
  \end{align*}
  Then $X$ can be compactly embedded into $\Lp{2}(\Omega)^{n}$.
\end{theorem}

\begin{proof}
  Let $(f_{n})_{n\in\N}$ be a bounded sequence in $X$, i.e.\ $\sup_{n\in\N}\norm{f_{n}}_{X} \leq K \in \R$.
  By assumption there exists a $\phi_{n} \in \Hspace^{1}_{\Gamma_{0}}(\Omega)$ such that $f_{n} = \grad \phi_{n}$
  for every $n\in\N$.
  By Poincar\'{e}'s inequality we have
  \begin{align*}
    \norm{\phi_{n}}_{\Hspace^{1}(\Omega)} \leq C \norm{\grad \phi_{n}}_{\Lp{2}(\Omega)} \leq C \norm{f_{n}}_{X}.
  \end{align*}
  Hence, $(\phi_{n})_{n\in\N}$ is a bounded sequence in $\Hspace^{1}(\Omega)$.
  Moreover, $(\boundtr \phi_{n})_{n\in\N}$ is a bounded sequence in $\Hspace^{\nicefrac{1}{2}}(\partial\Omega)$.
  By the compact embedding of $\Hspace^{1}(\Omega)$ into $\Lp{2}(\Omega)$ and $\Hspace^{\nicefrac{1}{2}}(\partial\Omega)$ into $\Lp{2}(\partial\Omega)$, there exists a subsequence $(\phi_{n(k)})_{k\in\N}$ that converges in $\Lp{2}(\Omega)$ such that also $(\boundtr \phi_{n(k)})_{k\in\N}$ converges in $\Lp{2}(\partial\Omega)$. W.l.o.g.\ we assume that this is already true for the original sequence. By
  \begin{align*}
    \MoveEqLeft\norm{f_{n}-f_{m}}_{\Lp{2}(\Omega)}^{2} \\
    &= \scprod{f_{n}-f_{m}}{\grad(\phi_{n}-\phi_{m})}_{\Lp{2}(\Omega)} \\
    &= -\scprod{\div(f_{n}-f_{m})}{\phi_{n}-\phi_{m}}_{\Lp{2}(\Omega)}
      %+ \scprod{\normaltr (f_{n}-f_{m})}{\boundtr(\phi_{n}-\phi_{m})}_{\Lp{2}(\Gamma_{1})}
       + \underbrace{\scprod{\normaltr (f_{n}-f_{m})}{\boundtr(\phi_{n}-\phi_{m})}_{\Lp{2}(\partial\Omega)}}_{\scprod{\normaltr (f_{n}-f_{m})}{\boundtr(\phi_{n}-\phi_{m})}_{\Lp{2}(\Gamma_{1})}}
    \\
    &\leq 2K\norm{\phi_{n}-\phi_{m}}_{\Lp{2}(\Omega)} + 2K \norm{\boundtr\phi_{n}-\boundtr\phi_{m}}_{\Lp{2}(\Gamma_{1})} \\
    &\to 0,
  \end{align*}
  we have that $(f_{n})_{n\in\N}$ is a Cauchy sequence in $\Lp{2}(\Omega)^{n}$ and therefore convergent.
\end{proof}

\begin{theorem}\label{le:dom-A-compact-embedding}
  $\dom (A)$ can be compactly embedded into $\XH$.
\end{theorem}

\begin{proof}
  Note that $\dom (A) \subseteq \XH$ and that $\norm{.}_{\XH}$ is equivalent to $\norm{.}_{\Lp{2}(\Omega)^{n+1}}$.
  We regard $\dom (A)$ with
  $\scprod{x}{y}_{A} = \scprod{x}{y}_{\XH} + \scprod{Ax}{Ay}_{\XH}$
  as inner product. Note that $\dom (A)$ is a Hilbert space with the previous inner product.
  The induced norm can be written as
  \begin{align*}
    \norm{x}_{A} =
    \sqrt{\norm{x}_{\XH}^{2}
    + \norm{T \grad \tfrac{1}{\rho} x_{1}}_{\Lp{2}}^{2}
    + \norm{\tfrac{1}{\rho}\div T x_{2}}_{\Lp{2}}^{2}}.
  \end{align*}
  Note that $\norm{\normaltr T x_{2}}_{\Lp{2}(\Gamma_{1})}$ is automatically bounded by $C\norm{x}_{A}$ for some $C>0$,
  since $\norm{\boundtr \tfrac{1}{\rho} x_{1}}_{\Hspace^{\nicefrac{1}{2}}(\partial\Omega)}$ is bounded by $C\norm{x}_{A}$ for some $C>0$
  and $\normaltr T x_{2}\big\vert_{\Gamma_{1}} = \tfrac{-1}{k}\boundtr \tfrac{1}{\rho} x_{1}\big\vert_{\Gamma_{1}}$.
  Let $X$ be the space from
  \Cref{th:compact-embedding-state-space}.
  Then
  \begin{align*}
    \Phi\colon \mapping{\dom (A)}{\Hspace^{1}_{\Gamma_{0}}(\Omega) \times X}{x}
    {\begin{bsmallmatrix}
        \frac{1}{\rho} & 0 \\
        0 & T
      \end{bsmallmatrix}x,}
  \end{align*}
  is continuous. Moreover, both $\Hspace^{1}_{\Gamma_{0}}(\Omega)$ and $X$ can be compactly embedded into $\Lp{2}(\Omega)$ and $\Lp{2}(\Omega)^{n}$, respectively. We denote this combined compact embeddeding by $\iota\colon \Hspace^{1}_{\Gamma_{0}}(\Omega) \times X \to \Lp{2}(\Omega)^{n+1}$. Hence, also $\dom (A)$ can be compactly embedded into $\XH$ by $\Phi^{-1}\iota\Phi$.
\end{proof}

\begin{corollary}\label{th:compact-resolvent-and-spectrum}
  The resolvent operators of $A$ are compact,   the spectrum of $A$ contains only eigenvalues and $\iu \R \cup \C_{+} \subseteq \uprho(A)$.
\end{corollary}

\begin{proof}
  By
  \Cref{le:dom-A-compact-embedding},
  $\dom (A)$ can be compactly embedded into $\XH$, which implies that every resolvent operator is compact. Hence, the spectrum of $A$ contains only eigenvalues.
  Since $0$ is not an eigenvalue by
  \Cref{th:zero-no-eigenvalue},
  we conclude that $0 \in \uprho(A)$.
  Moreover, by
  \Cref{th:i-axis-without-0-resolvent}
  also every other point on $\iu\R$ is in $\uprho(A)$.
\end{proof}

Finally we will prove
\Cref{th:semigroup-semi-uniform-stable}.

\begin{proof}[Proof of \Cref{th:semigroup-semi-uniform-stable}]%
  \label{prf:semigroup-semi-uniform-stable}
  By
  \Cref{th:compact-resolvent-and-spectrum}
  we have $\upsigma(A) \cap \iu \R = \emptyset$.
  Therefore, as announced in the beginning, \cite[Theorem~3.4]{chill-seifert-tomilov-2020} implies the semi-uniform stability of the semigroup generated by $A$.
\end{proof}

We conclude this section with an investigation of the strong stability of the operator $A_0$ given by~\eqref{eq:def-A_0}, which is an extension of $A$ and generates a strongly continuous semigroup on $\Lp{2}(\Omega)^{n+1}$.

\begin{lemma}\label{th:zero-eigenvalue-of-A0}
  Let $\Omega \subseteq \R^{n}$ be bounded and open with Lipschitz boundary,
  $n\geq 2$.
  Then the operator $A_{0}$ (defined in~\eqref{eq:def-A_0})
  has $\lambda = 0$ as an eigenvalue and thus, does not generate a strongly stable semigroup.
\end{lemma}

\begin{proof}
  Choose the components of
  $
  x =
  \left[
    \begin{smallmatrix}
      x_{1} \\ x_{2}
    \end{smallmatrix}
  \right]
  $
  as
  \begin{align*}
    x_{1} = 0
    \quad \text{and} \quad
    x_{2} = T^{-1}
    \begin{bsmallmatrix}
      \partial_{2} \phi \\ -\partial_{1} \phi \\ 0 \\ \smatvdots \\ 0
    \end{bsmallmatrix}
    ,
  \end{align*}
  where $\phi$ is any non zero $\sfC_{\mathrm{c}}^{\infty}(\Omega)$ function.
  % Since $\ran \rot \subseteq \ker \div$, we have $\div x_{2} = \div \rot \phi = 0$.
  Then $x_{2} \neq 0$ and $\div T x_{2} = \partial_{1}\partial_{2} \phi - \partial_{2}\partial_{1} \phi = 0$.
  Since $\phi$ has compact support, $x$ satisfies the boundary conditions. Thus $A_{0}$ cannot generate a strongly stable semigroup, since the eigenvector $x$ to $\lambda = 0$ is a constant solution of the Cauchy problem.
\end{proof}

\section{Conclusion}
In this paper we showed semi-uniform stability of the multidimensional wave equation equipped with a scattering passive feedback law. Further, we proved the the corresponding port-Hamiltonian operator has a compact resolvent.

To get compact embeddings for the port-Hamiltonian operator of the wave equation it is necessary to choose an adequate state space.
This is a new aspect that arises for spatial multidimensional port-Hamiltonian systems as in the one-dimensional spatial setting the compact embedding is always given.
It is likely that most of the techniques presented in this article will translate for general linear port-Hamiltonian systems on multidimensional spatial domains (see~\cite{skrepek-2020}) like Maxwell's equations and the Mindlin plate model. Probably the crucial tool will be a unique continuation principle.

Moreover, there is an interesting link between the resolvent set of the port-Hamiltonian operator of the wave equation and solvability of lossy Helmholtz equations. Since in the theory of Helmholtz equations (especially in view of finite element methods) a uniform bound of the solution operator is of interest, it might be possible to use results from that theory to give explicit decay rates for the semi-uniform stability or even obtain exponential stability under certain assumptions. For constant coefficients we can find such estimates in~\cite{melenk-diss,melenk-sauter-torres,gander-graham-spence-2015}.
There are some recent works on these estimates with non constant coefficients~\cite{graham-pembery-spence-2019,graham-sauter-2020}.

%\bibliography{bibfile}{}
%\addcontentsline{toc}{section}{References} % if references aren't added automatically to table of content
% \bibliographystyle{plain}
%\bibliographystyle{abbrv}
%\nocite{*}

\appendix

\section{Sobolev spaces}

We want to recall the most important relations between certain Sobolev spaces and boundary operators for our purpose. Details can be found in~\cite{monk}, \cite{dautray-lions-vol1}, \cite[ch.~IX]{dautray-lions-vol3} or in the appendix of~\cite{kurula-zwart-wave}. Let $\Omega \subseteq \R^{n}$ be open with bounded Lipschitz boundary. We define
\begin{align*}
  \Hspace^{1}(\Omega) &\coloneqq \dset{x \in \Lp{2}(\Omega)}{\grad x \in \Lp{2}(\Omega)^{n}},\\
  \Hspace(\div,\Omega) &\coloneqq \dset{x \in \Lp{2}(\Omega)^{n}}{\div x \in \Lp{2}(\Omega)},
\end{align*}
where $\div x$ and $\grad x$ are defined in a distributional sense. Here $\Hspace^{1}(\Omega)$ is equipped with the graph norm of $\grad$ and $\Hspace(\div, \Omega)$ is equipped with the graph norm of $\div$. By $\sfC_{\mathrm{c}}^{\infty}(\Omega)$ we denote the space of functions on $\Omega$ that are infinitely differentiable with compact support.
The restriction mapping $x \mapsto x\big\vert_{\partial \Omega}$ is defined for continuous functions on $\cl{\Omega}$.
There exists a continuous extension of this mapping from $\Hspace^{1}(\Omega)$ to $\Lp{2}(\partial\Omega)$. We define $\Hspace^{\nicefrac{1}{2}}(\partial\Omega)$ as the image of this mapping equipped with the range norm (and range inner product). Corresponding we define its (anti)dual space $\Hspace^{-\nicefrac{1}{2}}(\partial\Omega)$.
For $x, y\in \sfC_{\mathrm{c}}^{\infty}(\Omega)$ we have a integration by parts formula:
\begin{align*}
  \scprod{\div x}{y}_{\Lp{2}(\Omega)} + \scprod*{x}{\grad y}_{\Lp{2}(\Omega)}
  = \scprod*{\nu \cdot x\big\vert_{\partial\Omega}}{y\big\vert_{\partial\Omega}}_{\Lp{2}(\partial\Omega)}.
\end{align*}
This can be continuously extended to $x \in \Hspace(\div, \Omega)$ and $y \in \Hspace^{1}(\Omega)$, if we replace the $\Lp{2}(\partial\Omega)$ inner product by a dual pairing:
\begin{equation*}
  \scprod{\div x}{y}_{\Lp{2}(\Omega)} + \scprod*{x}{\grad y}_{\Lp{2}(\Omega)}\\
  = \dualprod{\normaltr x}{\boundtr y}_{\Hspace^{-\nicefrac{1}{2}}(\partial\Omega),\Hspace^{\nicefrac{1}{2}}(\partial\Omega)},
\end{equation*}
where $\boundtr \colon \Hspace^{1}(\Omega) \to \Hspace^{\nicefrac{1}{2}}(\partial\Omega)$ is the boundary trace (extension of $y \mapsto y\big\vert_{\partial\Omega}$) and $\normaltr \colon \Hspace(\div,\Omega) \to \Hspace^{-\nicefrac{1}{2}}(\partial\Omega)$ is the extension of $\nu \cdot x\big\vert_{\partial \Omega}$.
Furthermore,
\begin{align*}
  \Hspace^{1}_{\Gamma_{0}}(\Omega) \coloneqq
  \dset*{x \in \Hspace^{1}(\Omega)}{(\boundtr x)\big\vert_{\Gamma_{0}} = 0 \text{ in } \Lp{2}(\partial\Omega)}
\end{align*}
and $\Hspace^{\nicefrac{1}{2}}(\Gamma_{1})$ is defined as $\ran \boundtr\big\vert_{\Hspace^{1}_{\Gamma_{0}}(\Omega)}$ endowed with inner product from $\Hspace^{\nicefrac{1}{2}}(\partial\Omega)$ (for $g \in \Hspace^{\nicefrac{1}{2}}(\Gamma_{1})$ we can say that $g\big\vert_{\Gamma_{0}} = 0$).
We denote its (anti)dual space by $\Hspace^{-\nicefrac{1}{2}}(\Gamma_{1})$.
Then there is the following integration by parts formula for $x \in \Hspace(\div,\Omega)$ and $y\in \Hspace^{1}_{\Gamma_{0}}(\Omega)$
\begin{equation*}
  \scprod{\div x}{y}_{\Lp{2}(\Omega)} + \scprod{x}{\grad y}_{\Lp{2}(\Omega)}\\
  = \dualprod{\normaltr x}{\boundtr y}_{\Hspace^{-\nicefrac{1}{2}}(\Gamma_{1}),\Hspace^{\nicefrac{1}{2}}(\Gamma_{1})}.
\end{equation*}
We say $\normaltr x$ is in $\Lp{2}(\Gamma_{1})$, if there is an $f\in \Lp{2}(\Gamma_{1})$ such that
\begin{align*}
  \dualprod{\normaltr x}{\boundtr y}_{\Hspace^{-\nicefrac{1}{2}}(\Gamma_{1}),\Hspace^{\nicefrac{1}{2}}(\Gamma_{1})}
  = \scprod{f}{\boundtr y}_{\Lp{2}(\Gamma_{1})}
\end{align*}
for all $
  y \in \Hspace^{1}_{\Gamma_{0}}(\Omega)$.
Clearly in this case we say $\normaltr x = f$ and we can write the integration by parts formula as
\begin{align*}
  \scprod{\div x}{y}_{\Lp{2}(\Omega)} + \scprod{x}{\grad y}_{\Lp{2}(\Omega)}
  = \dualprod{\normaltr x}{\boundtr y}_{\Lp{2}(\Gamma_{1})}.
\end{align*}

\section{Solutions}

In this section we will discuss a suitable solution concept for~\eqref{eq:system-classical-formulation}. We will regard a solution $w(\cdot,\cdot)$ as a function in time mapping into spatial function space.
% For simplicity we will only regard initial data that are sufficiently regular.

An integrated version of the PDE is
\begin{align*}
  \rho(\zeta) \frac{\partial}{\partial t} w(t,\zeta) - \rho(\zeta) w_{1}(\zeta)
  = \int_{0}^{t} \div T(\zeta) \grad w(s,\zeta) \dx[s].
\end{align*}
We will demand that a solution will satisfy this integrated version of the PDE\@.

If we assume that both $\int_{0}^{t} \div T(\zeta) \grad w(s,\zeta) \dx[s]$ and $\div T(\zeta) \int_{0}^{t} \grad w(s,\zeta) \dx[s]$ exist, then they coincide and
\begin{align*}
  \rho(\zeta) \frac{\partial}{\partial t} w(t,\zeta) - \rho(\zeta) w_{1}(\zeta)
  =  \div T(\zeta) \int_{0}^{t} \grad w(s,\zeta) \dx[s].
\end{align*}
This is a consequence of the closedness of $\div$. For a classical solution ($w \in \sfC^{2}(\R_{+}\times\Omega) \cap \sfC^{1}(\R_{+} \times \cl{\Omega})$) these integrals coincide.

We will also regard an integrated version of the boundary conditions:
\begin{align*}
  \int_{0}^{t} \frac{\diffd}{\diffd s} w(s,\zeta) \dx[s]
  = -k \int_{0}^{t} \nu \cdot T \grad w(s,\zeta) \dx[s]
 \end{align*}
 for all $\zeta \in \Gamma_{1}$.
Again for classical solutions this can be manipulated to
\begin{align*}
  w(t,\zeta) - w(0,\zeta) &= -k \nu \cdot T \int_{0}^{t} \grad w(s,\zeta) \dx[s]
  \quad \text{for all} \quad \zeta \in \Gamma_{1}, \\
  \boundtr w(t,\cdot)\big\vert_{\Gamma_{1}} - \boundtr w(0,\cdot)\big\vert_{\Gamma_{1}}
  &= -k \normaltr \Big(T \int_{0}^{t} \grad w(s,\cdot) \dx[s] \Big)\Big\vert_{\Gamma_{1}}.
\end{align*}

\begin{definition}\label{def:solution}
  Let $w_{0} \in \Hspace^{1}(\Omega)$  and $w_{1} \in \Lp{2}(\Omega)$.
  Then we say that $w(\cdot,\cdot)$ is a \emph{solution} of~\eqref{eq:system-classical-formulation}, if
  $t \mapsto w(t,\cdot)$ is $\sfC^{1}(\R_{+};\Lp{2}(\Omega)) \cap \sfC^{0}(\R_{+};\Hspace^{1}(\Omega))$,
  and
  \begin{align*}
    \rho \frac{\diffd}{\diffd t} w(t,\cdot) - \rho w_{1}
    &=  \div T \int_{0}^{t} \grad w(s,\cdot) \dx[s], \\
    w(0,\cdot) &= w_{0}, \\
    \frac{\diffd}{\diffd t} w(t,\cdot)\Big\vert_{t=0} &= w_{1}, \\[1.2ex]
    \boundtr w(t,\cdot)\big\vert_{\Gamma_{0}} &= h, \\
    \boundtr w(t,\cdot)\big\vert_{\Gamma_{1}} - \boundtr w_{0}\big\vert_{\Gamma_{1}}
    % \boundtr \Big(\int_{0}^{t} \frac{\diffd}{\diffd s} w(s,\cdot) \dx[s]\Big)\Big\vert_{\Gamma_{1}}
    &= -k \normaltr \Big(T \int_{0}^{t} \grad w(s,\cdot) \dx[s]\Big)\Big\vert_{\Gamma_{1}},
  \end{align*}
  for all $t \in \R_{+}$.
\end{definition}

\begin{proposition}\label{th:equivalence-of-solutions}
  Let $w$ is a solution of~\eqref{eq:system-classical-formulation} in the sense of \Cref{def:solution} and $w_{\mathrm{e}}$ the solution of the equilibrium system~\eqref{eq:equilibrium-system}.
  Then
  \begin{align*}
    \begin{bmatrix}
      \rho \frac{\partial}{\partial t} w(t,\cdot) \\
      \grad w(t,\cdot) - \grad w_{\mathrm{e}}
    \end{bmatrix}
    \quad \text{and}\quad
    T(t)
    \begin{bmatrix}
      \rho w_{1} \\
      \grad w_{0} - \grad w_{\mathrm{e}}
    \end{bmatrix}
  \end{align*}
  coincide, where $T$ is the semigroup generated by $A$.

  On the other hand, let $x_{1}$ denote the first component of the solution given by the semigroup. Then
  \begin{align*}
    w(t,\cdot)\coloneqq \int_{0}^{t} \frac{1}{\rho} x_{1}(s) \dx[s] + w_{0} + w_{\mathrm{e}}
  \end{align*}
  is a solution of~\eqref{eq:system-classical-formulation} in the sense of \Cref{def:solution}.
\end{proposition}

\begin{remark}
  If we regard the semigroup $T_{0}$ generated by $A_{0}$, we can even cancel out $\grad w_{\mathrm{e}}$ and obtain
  \begin{align*}
        \begin{bmatrix}
      \rho \frac{\partial}{\partial t} w(t,\cdot) \\
      \grad w(t,\cdot)
    \end{bmatrix}
    =
    T_{0}(t)
    \begin{bmatrix}
      \rho w_{1} \\
      \grad w_{0}
    \end{bmatrix}
  \end{align*}
\end{remark}

\begin{theorem}
  The system~\eqref{eq:equilibrium-system} is solvable for $h \in \Hspace^{\nicefrac{1}{2}}(\Gamma_{0})$.
\end{theorem}

\begin{proof}
  Let $H \in \Hspace^{1}(\Omega)$ such that $h = \boundtr H \big\vert_{\Gamma_{0}}$.
  The weak formulation of~\eqref{eq:equilibrium-system} is: find a $\tilde{w} \in \Hspace^{1}_{\Gamma_{0}}(\Omega)$ such that
  \begin{align*}
    \scprod{\grad \tilde{w}}{\grad v}_{\Lp{2}(\Omega)} = -\scprod{\grad H}{\grad v}_{\Lp{2}(\Omega)}
     \end{align*}
  for all
   $v \in \Hspace^{1}_{\Gamma_{0}}(\Omega)$.
 Then $w_{\mathrm{e}} = \tilde{w} + H$. By the Lax-Milgram theorem this is solvable.
\end{proof}

\section{G{\aa}rding Inequalities}

In this section we want to show that there is a Fredholm alternative for sesquilinear forms that are non-coercive, but satisfy a \emph{G{\aa}rding inequality}.
In~\cite{yosida-fa} this concept is presented in a less abstract way for differential operators.
% Nevertheless, we present for general sesquilinear forms to highlight the elegance of this concept.

\begin{definition}\label{def:garding-inequality}
  Let $X_{0}$ and $X_{1}$ be Hilbert spaces and $K\colon X_{1} \to X_{0}$ be a compact linear operator.
  A sesquilinear form $b\colon X_{1} \times X_{1} \to \C$ satisfies a \emph{G{\aa}rding inequality}, if
  \begin{equation*}
    \Re b(u,u) \geq C_{1} \norm{u}_{X_{1}}^{2} - C_{2} \norm{K u}_{X_{0}}^{2}
    \quad\text{for all}\quad
    u \in X_{1}.
  \end{equation*}
\end{definition}

In most applications $K$ is a compact embedding, e.g.\ the embedding of $\Hspace^{1}(\Omega)$ into $\Lp{2}(\Omega)$. Note that (by Lax-Milgram, e.g.~\cite{evans-pdes}) for every bounded sesquilinear form $b(\cdot, \cdot)$ on a Hilbert space there exists a bounded operator $B\colon X_{1} \to X_{1}$ such that
\[
  b(u,v) = \scprod{B u}{v}_{X_{1}}
  \quad\text{for all}\quad
  u,v \in X_{1}.
\]
The operator $B$ is injective if and only if $b(\cdot,\cdot)$ is non-degenerated.

\begin{theorem}[Fredholm alternative]\label{th:fredholm-alternative-garding}
  Let $b(\cdot,\cdot)$ be a bounded sesquilinear form on $X_{1}$ that satisfies a G{\aa}rding inequality.
  If the corresponding operator $B$ is injective ($b(\cdot,\cdot)$ is non-degenerated), then $B$ is bijective.
\end{theorem}

\begin{proof}
  The sesquilinear form $b$ satisfies the G{\aa}rding inequality
  \[
    \Re b(u,u) \geq C_{1} \norm{u}_{X_{1}}^{2} - C_{2} \norm{K u}_{X_{0}}^{2}
    \quad\text{for all}\quad
    u \in X_{1}.
  \]
  Hence, $\tilde{b}(u,v) \coloneqq b(u,v) + C_{2} \scprod{Ku}{Kv}_{X_{0}}$ is coercive. The corresponding operator $\tilde{B}$ is given by $B + C_{2}K\adjun K$. By the Lax-Milgram theorem $\tilde{B}$ is bijective. Note that
  \[
    B = \tilde{B} - C_{2}K\adjun K = \tilde{B} ( \opid - \tilde{B}^{-1}C_{2}K\adjun K).
  \]
  The injectivity of $B$ implies that $1$ is not an eigenvalue of $\tilde{B}^{-1}C_{2}K\adjun K$ and since $\tilde{B}^{-1}C_{2}K\adjun K$ is compact, it is surjective. Consequently $B$ is also surjective.
\end{proof}

\vspace*{5mm}\hrule\vspace*{3mm}

\end{document}